\documentclass{article}

\usepackage{enumitem}
\usepackage{analysis}

\usepackage[doi=false, isbn=false,
            url=false, eprint=false,
            giveninits=true, maxbibnames=99]{biblatex}
\AtEveryBibitem{%
    \clearfield{month}
    \clearfield{day}
    \clearfield{note}
    \clearfield{pages}%
}
\DeclareUnicodeCharacter{0393}{\ensuremath{\Gamma}} 

\numberwithin{equation}{section}

\newcommand{\JJ}{\mathcal{J}}
\newcommand{\WW}{\mathcal{W}}
\DeclareMathOperator{\Graff}{Graff}

\title{Nonlocal Phase Transitions with Singular Heterogeneous Kernels}
\author{Wes Caldwell}
\date{\today}

\begin{document}

\maketitle

\begin{abstract}
    In this paper the study of a non-local Cahn-Hilliard-type singularly perturbed family of functionals is undertaken, generalizing known results by Alberti \& Bellettini \cite{alberti_non-local_1998}.
    The kernels considered include those leading to Gagliardo seminorms for fractional Sobolev spaces.
    The limit energy is computed via $\Gamma$-convergence and shown to be an anisotropic surface energy on the interface between the two phases.
\end{abstract}

\section{Introduction}

In \cite{alberti_non-local_1998}, Alberti \& Bellettini identify the $\Gamma$-limit of the functionals
\begin{equation}\label{eq:F}
    F_\varepsilon(u)
        \defeq \frac{1}{4\varepsilon} \int_\Omega \int_\Omega J_\varepsilon(y-x) \abs{u_\varepsilon(y) - u_\varepsilon(x)}^2 \,dy\,dx
        + \frac{1}{\varepsilon} \int_\Omega W(u) \,dx,
\end{equation}
where $\Omega \subset \RR^N$ is a bounded domain, $J_\varepsilon(h) \defeq \varepsilon^{-N}J(h/\varepsilon)$ for $J \colon \RR^N \to [0, +\infty)$ an even interaction potential satisfying
\begin{equation}\label{eq:abH1}
    \int_{\RR^N} J(h) \abs{h} \,dh < +\infty,
\end{equation}
and $W$ a continuous double-well potential which vanishes at $\pm 1$ only.

Such functionals arise, for example, as continuum limits of Ising spin systems on lattices; in this case, $u$ represents a macroscopic magnetization density (see \cite{alberti_surface_1996}, \cite{bricmont_surface_1980} and the references therein).
This model closely resembles the classical Cahn-Hilliard model for phase separation (see, e.g., \cite{baldo_minimal_1990}, \cite{bouchitte_singular_1990}, \cite{cahn_free_1958}, \cite{modica_gradient_1987}, \cite{sternberg_effect_1988}), given by the functional
\[
    E_\varepsilon(u) \defeq \varepsilon \int_\Omega \abs{\nabla u}^2 \,dx + \frac{1}{\varepsilon} \int_\Omega W(u) \,dx,
\]
which was originally treated in the context of $\Gamma$-convergence in the seminal paper of Modica \& Mortola \cite{modica_esempio_1977}.
As noted in \cite{alberti_non-local_1998}, the functionals $F_\varepsilon$ in (\ref{eq:F}) can be obtained from $E_\varepsilon$ by replacing the gradient term $\abs{\nabla u}^2$ with finite differences averaged with respect to $J$; that is,
\[
    \abs{\nabla u(x)}^2 \approx \int_{\RR^N} \frac{\abs{u(x+\varepsilon h) - u(x)}^2}{\varepsilon^2} J(h) \,dh.
\]

Functionals of the form $\mathcal{E}(u) \defeq \int_\Omega \int_\Omega J(y-x) \abs{u(y) - u(x)}^2 \,dy \,dx$ also arise in relation to symmetric L{\'e}vy processes (see, e.g., \cite{barlow_non-local_2006}, \cite{bertoin_levy_1998}).
To every measure $\nu$ on $\RR^N \setminus \set{0}$ satisfying $\int_{\RR^N}\left(\abs{h}^2 \wedge 1\right) \,d\nu(h) < +\infty$ we can associate a symmetric L{\'e}vy process with jumps distributed according to $\nu$; conversely, the L{\'e}vy-Khintchine formula (see \cite[Section 7.6]{chung_probability_1974} for more details) guarantees that every L{\'e}vy process has associated to it such a $\nu$.
In this context, the energy $\mathcal{E}$ arises as the quadratic form associated to the pseudo-differential operator $Lu(x) \defeq \int_{\RR^N} \left[u(x + h) - u(x)\right] J(h) \,dh$ (see \cite[Section 2]{foghem_gradient_2023}).

In \cite[Theorem 4.6]{alberti_nonlocal_1998}, Alberti \& Bellettini show that $F_\varepsilon$ is not identically equal to $+\infty$ if and only if
\begin{equation}\label{eq:abH2}
    \int_{\RR^N} J(h) \left(\abs{h} \wedge \abs{h}^2\right) < +\infty,
\end{equation}
leading them to claim in \cite{alberti_non-local_1998} that the hypothesis (\ref{eq:abH1}) can be relaxed to (\ref{eq:abH2}).
Crucially, the relaxed hypothesis (\ref{eq:abH2}) would allow for singular kernels $J(h) = \abs{h}^{-N-2s}$ for $\frac{1}{2} < s < 1$.
This choice of $J$ in (\ref{eq:F}) gives the (rescaled) ``fractional Allen-Cahn energy'' (see \cite{cabre_stable_2021}, \cite{savin_density_2014}),
\[
    E_\varepsilon^s(u) = \varepsilon^{2s-1} \abs{u}_{H^s(\Omega)}^2 + \frac{1}{\varepsilon} \int_\Omega W(u) \,dx,
\]
where $\abs{\cdot}_{H^s(\Omega)}$ is the fractional Sobolev seminorm in $\Omega$ (see \cite{leoni_first_2023} for background on fractional Sobolev spaces).
A variant of this functional including contributions outside of $\Omega$ was considered by Savin \& Valdinoci in \cite{savin_-convergence_2012}, where the $\Gamma$-limit was shown to be the same as the $\Gamma$-limit of the classical Cahn-Hilliard functional $E_\varepsilon$.

In this paper, we will consider the functional (\ref{eq:F}) under the weaker hypothesis (\ref{eq:abH2}), and prove that we still recover the same $\Gamma$-limit as in \cite{alberti_non-local_1998}.
Unless stated otherwise, we will always assume the following hypotheses:
\begin{enumerate}[label=\textit{(\roman*)}]
    \item $J \colon \RR^N \to [0, +\infty)$ is an even, measurable function satisfying 
        \begin{align}
            \int_{\RR^N} J(h) \left(\abs{h} \wedge \abs{h}^2\right) \,dh \eqdef M_J < +\infty, \label{H1}\tag{H1} \\
            \int_{B_1^c} J(h) \abs{h} \log\abs{h} \,dh < + \infty, \label{H2}\tag{H2}
        \end{align}
        where $B_r$ denotes the ball of radius $r$, and the superscript $\,^c$ denotes the complement in $\RR^N$.
    \item
        $W \colon \RR \to [0, +\infty)$ is a continuous function which vanishes at $\pm 1$ only, and has at least linear growth at infinity.
        That is, there exist $C, R > 0$ such that
        \[
            W(z) \geq C\abs{z} \text{ for all } \abs{z} \geq R.
        \]
\end{enumerate}
        
\subsection{Main Theorem}

We construct the anisotropic limit functional in a standard way using cell formulae (cf. \cite{alberti_non-local_1998}, \cite{dal_maso_asymptotic_2017}).
Let $\mathcal{F}$ be the unscaled functional defined by
\[
    \mathcal{F}(u, A)
        \defeq \frac{1}{4} \int_A \int_{\RR^N} J(h) \abs{u(x + h) - u(x)}^2 \,dh \,dx
        + \int_A W(u(x)) \,dx,
\]
for all open $A \subset \RR^N$ and $u \colon \RR^N \to \RR$ measurable.
For a fixed unit vector $\xi \in S^{N-1}$, let $\mathcal{C}_\xi$ be the collection of cubes of dimension $N-1$ centered at the origin in the hyperplane $\xi^\perp$. 
Furthermore, for each cube $C \in \mathcal{C}_\xi$ define via Minkowski sum the strip $T_C \defeq C + \RR\xi = \set{x + t\xi \given x \in C, t \in \RR}$, and the class of functions
\[
    X(C) \defeq \set{u \colon \RR^N \to [-1, 1] \given u \text{ is $C$-periodic}, \lim_{\inner{x}{\xi} \to \pm \infty} u(x) = \pm 1}.
\]
We define the anisotropic surface tension $\psi$ as
\begin{equation}\label{eq:psi}
    \psi(\xi) \defeq \inf \set{\abs{C}^{-1} \mathcal{F}(u, T_C) \given C \in \mathcal{C}_\xi, u \in X(C)},
\end{equation}
where $\abs{C}$ is the Lebesgue measure of $C$.
The function $\psi$ is upper semicontinuous on the sphere (cf. \cite[Lemma 5.3]{alberti_non-local_1998}).

The limit functional $F$ is thus defined as
\begin{equation}
    F(u) \defeq \begin{cases}
        \int_{S_u} \psi(\nu_u) \,d\Hausdorff^{N-1} &\textup{ if } u \in BV(\Omega; \set{-1, 1}), \\
        +\infty &\textup{ if } u \in L^1(\Omega) \setminus BV(\Omega; \set{-1, 1}).
    \end{cases}
\end{equation}
For definitions of $S_u, \nu_u, \Hausdorff^{N-1}$, and $BV(\Omega; \set{-1, 1})$, see Section 1.3.
We now state our main theorem.

\begin{theorem}\label{thm:main}
    Let $\varepsilon_j \to 0$ as $j \to \infty$.
    Under the hypotheses (i) and (ii) on $J$ and $W$, the following statements hold:
    \begin{enumerate}[ref=\thetheorem(\arabic*)]
        \item\label{item:compactness} \textup{Compactness:}
            Let $\set{u_j} \subset L^1(\Omega)$ be such that
            \[
                \sup_{j \in \NN} F_{\varepsilon_j}(u_j) < +\infty.
            \]
            Then there exists a subsequence $\{\varepsilon_{j_n}\}_{n \in \NN}$ such that $u_{j_n} \to u$ in $L^1(\Omega)$ for some $u \in BV(\Omega; \set{-1, 1})$.

        \item\label{item:liminf} \textup{$\Gamma$-lim inf}:
            For all $u_j \to u$ in $L^1(\Omega)$, then we have
            \[
                \liminf_{j \to \infty} F_{\varepsilon_j}(u_j) \geq F(u).
            \]

        \item\label{item:limsup} \textup{$\Gamma$-lim sup}:
            For all $u \in BV(\Omega; \set{-1, 1})$, there exists a sequence $\{u_j\} \subset L^1(\Omega)$ such that $u_j \to u$ in $L^1(\Omega)$ and
            \[
                \limsup_{j \to \infty} F_{\varepsilon_j}(u_j) \leq F(u).
            \]
    \end{enumerate}
\end{theorem}

Parts 2 and 3 of Theorem \ref{thm:main} can be summarized by saying that the sequence $\set{F_\varepsilon}$ $\Gamma$-converges to $F$ in $L^1(\Omega)$.
For more information on $\Gamma$-convergence, see \cite{dal_maso_introduction_1993}.

As indicated in \cite{alberti_non-local_1998}, a truncation and diagonalization argument can be used with their results to prove the compactness and $\Gamma-\liminf$ results under our hypotheses (see Section 2 for more details).
However, their arguments cannot directly be used to prove the $\Gamma-\limsup$ under our hypotheses.
When $J$ is allowed to have a very strong singularity at the origin, boundedness of $F_\varepsilon(u)$ enforces regularity of $u$, so a more sophisticated gluing procedure is needed to construct smooth recovery sequences.

\subsection{Discussion of Hypotheses}

Hypothesis (\ref{H2}) is novel, and is needed in the proof and use of the modification theorem, Theorem \ref{thm:mod}.
As indicated in \cite{alberti_non-local_1998}, their methods can identify the $\Gamma-\liminf$ of $F_\varepsilon$ under hypothesis (\ref{H1}) since one can approximate $J$ from below by a sequence of kernels truncated away from the origin. However, this direct approximation does not work for the $\Gamma-\limsup$ since one would need to approximate the singularity from above, which is not possible.

Note further that the far-field bounds in (\ref{H1}) follow directly from the second hypothesis (\ref{H2}).
We choose to state the hypotheses in this overlapping manner since the quantity $M_J$ defined in (\ref{H1}) is useful and because it is unclear at present whether hypothesis (\ref{H2}) is strictly necessary for Theorem \ref{thm:main} or is merely a technical necessity limited by the methods used.

To better quantify the rate of decay of $J(h)\abs{h}$ at infinity, define the function $\omega_1 \colon [0, +\infty) \to [0, +\infty)$ by
\begin{equation}\label{eq:omega1}
    \omega_1(t)
        \defeq \int_{B_{1/t}^c} J(h)\abs{h} \,dh.
\end{equation}
If $J$ satisfies hypothesis (\ref{H1}), then certainly $\omega_1(t) \to 0$ as $t \to 0^+$.
The following proposition shows that hypothesis (\ref{H2}) can also be stated in terms of $\omega_1$ alone.

\begin{proposition}\label{prop:hyp}
    Suppose that $J$ satisfies hypothesis (\ref{H1}).
    Then (\ref{H2}) holds if and only if
    \begin{equation}
        \sum_{n = 1}^\infty \omega_1\left(\frac{1}{n}\right)\frac{1}{n} < +\infty. \label{H2*} \tag{H2*}
    \end{equation}
\end{proposition}

\begin{proof}
    Note that $t \mapsto \omega_1\left(\frac{1}{t}\right)\frac{1}{t}$ is a decreasing function of $t$ for $t > 1$, so (\ref{H2*}) is satisfied if and only if $\omega_1(1) < +\infty$ and
    \[
        \int_1^\infty \omega_1\left(\frac{1}{t}\right)\frac{1}{t} \,dt < +\infty.
    \]
    Since $J$ satisfies (\ref{H1}), we have that $\omega_1(1) \leq M_J < +\infty$.
    Now, letting $\chi_U$ denote the indicator function of a set $U \subset \RR^N$ and applying Fubini's theorem, we get
    \begin{align*}
        \int_1^\infty \omega_1\left(\frac{1}{t}\right)\frac{1}{t} \,dt
            &= \int_1^\infty \frac{1}{t} \int_{B_t^c} J(h) \abs{h} \,dh \,dt \\
            &= \int_1^\infty \frac{1}{t} \int_{\RR^N} \chi_{\set{\abs{h} \geq t}} J(h) \abs{h} \,dh \,dt \\
            &= \int_{\RR^N} \int_1^\infty \frac{1}{t} \chi_{\set{\abs{h} \geq t}} \,dt J(h) \abs{h} \,dh \\
            &= \int_{B_1^c} \int_1^{\abs{h}} \frac{1}{t} \,dt J(h) \abs{h} \,dh \\
            &= \int_{B_1^c} J(h) \abs{h} \log\abs{h} \,dh,
    \end{align*}
    completing the proof.
\end{proof}

This equivalent hypothesis is clearly satisfied by any kernel with asymptotic decay $\omega_1(t) = \BigO(t^\alpha)$ for some $\alpha > 0$.
As an example, the fractional Sobolev kernel $J(h) = \abs{h}^{-N-2s}$ has associated decay $\omega_1(t) \sim t^{2s-1}$, so satisfies (\ref{H2*}) if and only if $s > 1/2$.
These same asymptotics also hold for anisotropic Sobolev kernels $J_K(h) \defeq \norm{h}_K^{-N-2s}$, where $\norm{\cdot}_K$ is an arbitrary norm on $\RR^N$ with unit ball $K$ (cf. \cite{ludwig_anisotropic_2014}), so $J_K$ also satisfies hypotheses (\ref{H1}) and (\ref{H2}).

\subsection{Sets of Finite Perimeter and Functions of Bounded Variation}

In this section we recall the definition and some basic theory of sets of finite perimeter and functions of bounded variation (see, e.g., \cite[Section 5]{evans_measure_1992}, \cite{maggi_sets_2012}).

\begin{definition}
    A set $E \subset \RR^N$ such that $\abs{E} < +\infty$ is a set of finite perimeter if
    \[
        P(E) \defeq \sup \set*{\int_E \Div \varphi \,dx \given \varphi \in C_c^1(\RR^N; \RR^N), \norm{\varphi}_\infty \leq 1} < +\infty.
    \]
\end{definition}

Finiteness of $P(E)$ implies that the distributional derivative of the indicator function $1_E$ is a vector-valued Radon measure such that $\abs{D1_E}(\RR^N) = P(E)$, and for all $\varphi \in C_c^1(\RR^N)$,
\begin{equation}\label{eq:dderiv}
    \int_{\RR^N} \varphi \cdot \,dD1_E = \int_E \Div \varphi \,dx.
\end{equation}

We let $BV(\Omega)$ denote the space of all functions $u \colon \Omega \to \RR$ of bounded variation; that is, functions $u \in L^1(\Omega)$ whose distributional derivatives are vector-valued Radon measures.
Furthermore, we let $BV(\Omega; \set{-1, 1})$ denote the subspace of functions of bounded variation taking values $\pm 1$ only.
We may identify the space $BV(\Omega; \set{-1,1})$ with sets of finite perimeter, since every function $u \in BV(\Omega; \set{-1, 1})$ is uniquely determined by the set $\set{u = +1}$.

\begin{definition}
    Let $E \subset \RR^N$ be a set of finite perimeter.
    We define the \emph{reduced boundary} $\partial^*E$ of $E$ to be the set of points $x \in \RR^N$ for which the limit
    \[
        \nu_E(x) = -\lim_{r \to 0^+} \frac{D1_E(x + B_r)}{\abs{D1_E}(x + B_r)},
    \]
    exists and satisfies $\abs{\nu_u(x)} = 1$, where $B_r$ is the ball centered at the origin of radius $r$.
    We call the vector $\nu_u(x)$ the \emph{exterior measure-theoretic normal to $E$} at $x$.
\end{definition}

De Giorgi's Structure Theorem (see, e.g., \cite[Theorem 15.9]{maggi_sets_2012}) relates the measure $D1_E$ to the Hausdorff measure $\Hausdorff^{N-1}$; specifically, we have that $\partial^*E$ is an $\Hausdorff^{N-1}$-rectifiable set and $D1_E = -\nu_E \Hausdorff^{N-1} \mres \partial^*E$.
In other words, we may rewrite (\ref{eq:dderiv}) as the \textit{generalized Gauss-Green Theorem}
\[
    \int_{\partial^*E} \varphi \cdot \nu_E \,d\Hausdorff^{N-1} = \int_E \Div \varphi \,dx.
\]

Since every $u \in BV(\Omega; \set{-1, 1})$ can be identified with $\set{u = +1}$, we define the ``jump set'' $S_u \defeq \partial^*{\set{u = +1}}$ and the exterior normal $\nu_u \defeq \nu_{\set{u = +1}}$.

\subsection{Slicing and Integral Geometry}

Let $S^{N-1} \subset \RR^N$ be the unit sphere, and let $\Graff(N,1)$ be the affine Grassmannian of lines in $\RR^N$ (for background on $\Graff(N,k)$, see \cite[Chapter 6]{klain_introduction_1997}).
$\Graff(N,1)$ supports a rigid motion invariant Haar measure, which (suitably normalized) we will denote by $\lambda_1^N$.
This measure can be described explicitly: given an affine line $L \in \Graff(N,1)$, we may parameterize $L = x_0 + \RR\xi$, where $\xi \in S^{N-1}$ is either of the unit vectors spanning $L$, and $x_0 \in \xi^\perp$ is a basepoint lying on the hyperplane orthogonal to $\xi$.
Given a measurable function $f \colon \Graff(N,1) \to \RR$, we therefore define
\begin{equation}\label{eq:lambdadef}
    \int_{\Graff(N,1)} f(L) \,d\lambda_1^N(L)
        \defeq \frac{1}{2} \int_{S^{N-1}} \int_{\xi^\perp} f(x_0 + \RR\xi) \,d\Hausdorff^{N-1}(x_0) \,d\Hausdorff^{N-1}(\xi).
\end{equation}
The following lemma lets us relate the integral over a function of two spatial variables to an integral over its one dimensional slices.

\begin{lemma}[{Blaschke-Petkantschin Formula, c.f. \cite[Theorem 6.46]{leoni_first_2023}, \cite[Theorem 7.2.7]{schneider_stochastic_2008}}]\label{lem:abp}
    Let $f \colon A \times B \to \RR$ be a measurable function, $A,B \subset \RR^N$ measurable subsets.
    Then,
    \begin{equation}\label{eq:abp}
        \begin{split}
            \int_A \int_B &f(x,y) \,dy \,dx \\
                &= \int_{\Graff(N,1)} \int_{A \cap L} \int_{B \cap L} f(x,y)\abs{y-x}^{N-1} \,d\Hausdorff^1(y) \,d\Hausdorff^1(x) \,d\lambda_1^N(L).
        \end{split}
    \end{equation}
\end{lemma}

Given a set $E \subset \RR^N$ and $\delta > 0$, define the ``inner set''
\[
    E_\delta \defeq \set{x \in E \given d(x, \partial E) > \delta},
\]
and the ``outer set''
\[
    E^\delta \defeq \set{x \in \RR^N \given d(x, E) < \delta}.
\]
In particular, we may write the outer set as the Minkowski sum $E^\delta = E + B_\delta$, where $B_\delta$ is the ball centered at the origin of radius $\delta$.
The following lemma lets us control the volume of the outer set in terms of its $\Hausdorff^{N-1}$ measure.

\begin{lemma}\label{lem:steiner}
    Suppose $X \subset \RR^N$ is a finite union of compact, convex sets of dimension $N-1$ with pairwise $\Hausdorff^{N-1}$-null intersections; that is, $X = \bigcup_{k=1}^K X_k$ for some $K \in \NN$, where each $X_k$ lies in a hyperplane, and $\Hausdorff^{N-1}(X_i \cap X_j) = 0$ for $i \neq j$.
    If $\delta > 0$, then
    \begin{equation}\label{eq:steiner}
        \Hausdorff^N(X^\delta)
            \leq 2\pi \delta \Hausdorff^{N-1}(X) + \BigO(\delta^2).
    \end{equation}
\end{lemma}

\begin{proof}
    First, note that
    \[
        X^\delta = \bigcup_{k=1}^K X_k^\delta,
    \]
    so that we may use subadditivity of the Hausdorff measure to get
    \[
        \Hausdorff^N(X^\delta)
            \leq \sum_{k = 1}^K \Hausdorff^N(X_k^\delta).
    \]
    To bound $\Hausdorff^N(X_k^\delta)$, we will appeal to Steiner's Formula \cite[Theorem 9.2.3]{klain_introduction_1997}: for all compact, convex sets $E \subset \RR^N$,
    \begin{equation}\label{eq:steinersformula}
        \mu_N(E^\delta) = \sum_{i=1}^N \mu_i(E)\omega_{N-i} \delta^{N-i},
    \end{equation}
    where $\mu_i$ are the ``intrinsic volumes'' on $\RR^N$ and $\omega_i$ is the volume of the $i$-dimensional unit ball (see \cite{klain_introduction_1997} for more information on intrinsic volumes).
    In particular, $\mu_i(E) = \Hausdorff^i(E)$ if $E$ lies within a subspace of dimension $\leq i$.

    Therefore, since each $X_k$ lies within a subspace of dimension $N-1$ (a hyperplane), we have $\mu_N(X_k) = \Hausdorff^N(X_k) = 0$ and $\mu_{N-1}(X_k) = \Hausdorff^{N-1}(X_k)$.
    Applying (\ref{eq:steinersformula}) to $\Hausdorff^N(X_k^\delta)$, we find
    \begin{align*}
        \Hausdorff^N(X^\delta)
            &\leq \sum_{k = 1}^K \Hausdorff^N(X_k^\delta) \\
            &= 2\pi\delta\sum_{k = 1}^K \Hausdorff^{N-1}(X_k) + \BigO(\delta^2) \\
            &= 2\pi\delta\Hausdorff^{N-1}(X) + \BigO(\delta^2),
    \end{align*}
    where the last equality follows since the $X_k$ have pairwise $\Hausdorff^{N-1}$-null intersections.
\end{proof}

\section{Compactness and the $\Gamma-\liminf$}

Given a kernel $J$ satisfying hypotheses (\ref{H1}) and (\ref{H2}), for $0 < \rho < 1$ define the truncated kernels
\[
    J^\rho(h) \defeq \mathbf{1}_{B_\rho^c}(h) J(h).
\]
That is, $J^\rho(h) = 0$ if $\abs{h} \leq \rho$ and $J^\rho(h) = J(h)$ otherwise.
Furthermore, denote with a superscript $\rho$ the variants of $J_\varepsilon, F_\varepsilon, F$, etc., respectively defined with $J^\rho$ in place of $J$.
We prove the compactness statement in Theorem \ref{item:compactness} by considering a single truncation.

\begin{proof}[Proof of Theorem \ref{item:compactness}]
    Let $\varepsilon_j \to 0$, and $\{u_j\} \subset L^1(\Omega)$ be such that $F_{\varepsilon_j}(u_j)$ is uniformly bounded.
    Choose any $0 < \rho < 1$.
    Since $F_\varepsilon^\rho \leq F_\varepsilon$ for all $\varepsilon > 0$, $F_{\varepsilon_j}^\rho(u_j)$ is also uniformly bounded.
    The kernel $J^\rho$ satisfies the hypotheses of \cite[Theorem 1.4]{alberti_non-local_1998}, so we may apply the theorem to conclude that there exists a subsequence of $u_j$ converging to some $u \in BV(\Omega; \set{-1, 1})$.
\end{proof}

In order to prove the $\Gamma-\liminf$ statement in Theorem \ref{item:liminf}, it will be necessary to approximate $J$ by the truncated kernels $J^\rho$ and sending $\rho \to 0^+$.
The following lemmas ensure that this approximation process properly recovers the surface tension and energies, respectively, associated to the kernel $J$.

\begin{lemma}\label{lem:psi}
    Let $\xi \in S^{N-1}$.
    Then
    \[
        \lim_{\rho \to 0} \psi^\rho(\xi)
            = \sup_{\rho > 0} \psi^\rho(\xi)
            = \psi(\xi).
    \]
\end{lemma}

\begin{proof}
    Recall that $\psi^\rho$ and $\psi$ are defined as
    \begin{align}
        \psi^\rho(\xi) &\defeq \inf_{C \in \mathcal{C}_\xi, u \in X(C)} \abs{C}^{-1} \mathcal{F}^\rho(u, T_C), \label{eq:psiro} \\
        \psi(\xi) &\defeq \inf_{C \in \mathcal{C}_\xi, u \in X(C)} \abs{C}^{-1} \mathcal{F}(u, T_C).
    \end{align}
    By definition, for all $x \in \RR^N \setminus \set{0}$, the function $\rho \mapsto J^\rho(x)$ is decreasing and convergences to $J(x)$ as $\rho \to 0^+$.
    Therefore, the functions $\rho \mapsto \abs{C}^{-1} \mathcal{F}^\rho(u, T_C)$ are decreasing, so the equivalence of the limit and supremum is clear.
    All that remains to show is that $\sup_{\rho > 0} \psi^\rho(\xi) \geq \psi(\xi)$.

    Assume that $\sup_{\rho > 0} \psi^\rho(\xi) < +\infty$, since otherwise there is nothing to prove.
    By \cite[Theorem 2.4, Theorem 3.3]{alberti_nonlocal_1998}, for each $\rho > 0$, the infimum in (\ref{eq:psiro}) is independent of $C$ and achieved by a function $u^\rho \colon \RR^N \to [-1,1]$ given by $u^\rho(x) \defeq \gamma_\xi^\rho(\inner{x}{\xi})$ for some non-decreasing function $\gamma_\xi^\rho \colon \RR \to [-1,1]$.
    Therefore, we may apply Helly's theorem to the family $\set{\gamma_\xi^\rho}_{\rho > 0}$ to extract a subsequence $\rho_j \to 0$ such that $\gamma_\xi^{\rho_j} \to \gamma_\xi$ pointwise for some non-decreasing $\gamma_\xi \colon \RR \to [-1,1]$.
    Define $u \colon \RR^N \to \RR$ by $u(x) \defeq \gamma_\xi(\inner{x}{\xi})$, and let $C \in \mathcal{C}_\xi$ be arbitrary.
    Then $u \in X(C)$, and we can conclude by Fatou's lemma,
    \begin{align*}
        \sup_{\rho > 0} \psi^\rho(x)
            &= \lim_{j \to \infty} \abs{C}^{-1} \mathcal{F}^{\rho_j}(u^{\rho_j}, T_C) \\
            &\geq \abs{C}^{-1} \mathcal{F}(u, T_C) \\
            &\geq \psi(\xi).
    \end{align*}
\end{proof}

\begin{lemma}\label{lem:trunc}
    Let $\varepsilon > 0$, and $u \in L^1(\Omega)$.
    Then
    \begin{align*}
        \lim_{\rho \to 0} F_\varepsilon^\rho(u)
            &= \sup_{\rho > 0} F_\varepsilon^\rho(u)
             = F_\varepsilon(u), \\
        \lim_{\rho \to 0} F^\rho(u)
            &= \sup_{\rho > 0} F^\rho(u)
             = F(u).
    \end{align*}
\end{lemma}

\begin{proof}
    As in the proof of Lemma \ref{lem:psi}, $F_\varepsilon^\rho$ and $F^\rho$ are decreasing in $\rho$ so the equivalence of the limit and the supremum follows.
    We may then simply pass the supremum under the integral by the monotone convergence theorem to conclude the first limit, and use Lemma \ref{lem:psi} with the monotone convergence theorem to conclude the second limit.
\end{proof}

With these lemmas, the proof of the $\Gamma-\liminf$ is now straightforward.

\begin{proof}[Proof of Theorem \ref{item:liminf}]
    Let $\varepsilon_j \to 0$, and $u_j \to u$ in $L^1(\Omega)$.
    We may suppose, without loss of generality, that
    \[
        \liminf_{j \to \infty} F_{\varepsilon_j}(u_j) < +\infty.
    \]
    By Theorem \ref{item:compactness}, $u \in BV(\Omega; \set{-1, 1})$, so we may apply \cite[Theorem 1.4]{alberti_non-local_1998} to conclude that for each $0 < \rho < 1$,
    \[
        \liminf_{j \to \infty} F_{\varepsilon_j}^\rho(u_j) \geq F^\rho(u).
    \]
    Combining this with Lemma \ref{lem:trunc} and the monotone convergence theorem, we compute
    \begin{align*}
        \liminf_{j \to \infty} F_{\varepsilon_j}(u_j)
            &= \liminf_{j \to \infty} \sup_{\rho > 0} F_{\varepsilon_j}^\rho(u_j) \\
            &\geq \sup_{\rho > 0} \liminf_{j \to \infty} F_{\varepsilon_j}^\rho(u_j) \\
            &\geq \sup_{\rho > 0} F^\rho(u)
             = F(u).
    \end{align*}
\end{proof}

\section{Modification Theorem} 

In the proof of the $\Gamma-\limsup$ inequality, we will need to patch together different optimal profiles corresponding to different directions.
In the case where $J$ has only a weak singularity at the origin, it is sufficient to do a na{\"i}ve gluing to construct a recovery sequence.
If $J$ has a strong enough singularity, then boundedness of $F_\varepsilon(u)$ enforces regularity of $u$ and a more sophisticated process is required.
The modification theorem allows us to use a fixed profile to patch continuously between the optimal profiles of different directions in a manner that does not affect $F_\varepsilon$ energetically too much.

\begin{theorem}[Modification Theorem]\label{thm:mod}
    Let $\delta > 0$.
    Let $D \subset \Omega$, $\set{u_k} \subset L^1(D; [-1, 1])$, and $\set{w_k} \subset L^1(\Omega \setminus D_\delta; [-1, 1])$ such that $u_k - w_k \to 0$ in $L^1(D \setminus D_\delta)$.
    Then there exists $\set{v_k} \subset L^p(\Omega)$ such that $v_k = u_k$ in $D_\delta$, $v_k = w_k$ in $\Omega \setminus D$, and
    \begin{equation}\label{eq:modification}
        \limsup_{k \to \infty} F_{\varepsilon_k}(v_k, \Omega)
        \leq \limsup_{k \to \infty} \left[F_{\varepsilon_k}(u_k, D) + F_{\varepsilon_k}(w_k, \Omega \setminus D_\delta)\right].
    \end{equation}
\end{theorem}

The proof of this theorem follows exactly the structure of Proposition 4.1 in \cite{savin_-convergence_2012}, with additional care and many ideas from \cite{dal_maso_asymptotic_2017} taken to treat the case of a general kernel.
As such, we will need the following lemmas, adapted from \cite{dal_maso_asymptotic_2017}.
Lemma \ref{lemma:interior} will address near field interactions for $\JJ_\varepsilon$, and Lemma \ref{lemma:separation} will address far field interactions for $\JJ_\varepsilon$, where for measurable subsets $A,B \subset \Omega$, we define
\begin{align}
    F_\varepsilon(u, A)
        &\defeq \JJ_\varepsilon(u, A) + \WW_\varepsilon(u, A), \nonumber \\
    \JJ_\varepsilon(u, A)
        &\defeq \JJ_\varepsilon(u, A, A), \nonumber \\
    \JJ_\varepsilon(u, A, B)
        &\defeq \frac{1}{4\varepsilon} \int_A \int_B J_\varepsilon(y-x)\abs{u(y) - u(x)}^2 \,dy\,dx, \label{eq:J} \\
    \WW_\varepsilon(u, A)
        &\defeq \frac{1}{\varepsilon} \int_A W(u(x)) \,dx. \nonumber
\end{align}

\begin{lemma}[{cf. \cite[Lemma~5.3]{dal_maso_asymptotic_2017}}]\label{lemma:interior}
    Let $\varepsilon > 0$, let $y \in \RR^N$, let $A$ be a measurable subset of $\RR^N$, and let $g \colon A \to \RR$ be a measurable function such that
    \[
        0 \leq g(x) \leq \left((a\abs{y-x}) \wedge b\right)^2 \quad \text{ for every } x \in A,
    \]
    for some constants $a,b > 0$.
    If $b \geq a\varepsilon$, then
    \[
        \int_A J_\varepsilon(y-x)g(x) \,dx
            \leq abM_J \varepsilon,
    \]
    where $M_J$ is defined as in (\ref{H1}).
\end{lemma}

\begin{proof}
    Using the change of variables $h \defeq (y-x)/\varepsilon$, we get
    \begin{align*}
        \int_A J_\varepsilon(y-x)g(x) \,dx
            &\leq \int_{\RR^N} J(h) \left((a\varepsilon\abs{h}) \wedge b\right)^2 \,dh \\
            &= a^2 \varepsilon^2 \int_{B_{b/(a\varepsilon)}} J(h) \abs{h}^2 \,dh \\
            &\qquad+ b^2 \int_{\RR^N \setminus B_{b/(a\varepsilon)}} J(h) \,dh \\
            &= a^2 \varepsilon^2 \int_{B_1} J(h) \abs{h}^2 \,dh
             + a^2 \varepsilon^2 \int_{B_{b/(a\varepsilon)} \setminus B_1} J(h) \abs{h}^2 \,dh \\
            &\qquad+ b^2 \int_{\RR^N \setminus B_{b/(a\varepsilon)}} J(h) \,dh \\
            &\leq a^2 \varepsilon^2 \int_{B_1} J(h) \abs{h}^2 \,dh
             + ab \varepsilon \int_{\RR^N \setminus B_1} J(h) \abs{h} \,dh \\
            &\leq abM_J \varepsilon.
    \end{align*}
\end{proof}

\begin{lemma}[{cf. \cite[Lemma~5.4]{dal_maso_asymptotic_2017}}]\label{lemma:separation}
    Let $\varepsilon > 0$, $\delta > 0$, let $E, F$ be open sets in $\RR^N$ with $d(E,F) \geq \delta$, and let $u \in L^2(E \cup F)$.
    Then,
    \[
        \JJ_\varepsilon(u, E, F)
            \leq \frac{1}{2\delta} \omega_1\left(\frac{\varepsilon}{\delta}\right) \int_{E \cup F} \abs{u(x)}^2 \,dx,
    \]
    where $\omega_1$ is as defined in (\ref{eq:omega1}).
\end{lemma}

\begin{proof}
    Using the change of variables $h = (y-x)/\varepsilon$ and the fact that $J$ is even, we have
    \begin{align*}
        \JJ_\varepsilon(u, E, F)
            &= \frac{1}{4\varepsilon} \int_E \int_F J_\varepsilon(y-x) \abs{u(y) - u(x)}^2 \,dy\,dx \\
            &\leq \frac{1}{2\varepsilon} \int_E \int_F J_\varepsilon(y-x) \,dy \abs{u(x)}^2 \,dx \\
            &\quad+ \frac{1}{2\varepsilon} \int_F \int_E J_\varepsilon(x-y) \,dx \abs{u(y)}^2 \,dy \\
            &\leq \frac{1}{2\varepsilon} \int_E \int_{\RR^N \setminus B_{\delta/\varepsilon}} J(h) \,dh \,\abs{u(x)}^2 \,dx \\
            &\quad+ \frac{1}{2\varepsilon} \int_F \int_{\RR^N \setminus B_{\delta/\varepsilon}} J(h) \,dh \,\abs{u(y)}^2 \,dy \\
            &= \frac{1}{2\varepsilon} \left(\int_{\RR^N \setminus B_{\delta/\varepsilon}} J(h) \,dh\right) \left(\int_{E \cup F} \abs{u(x)}^2 \,dx\right) \\
            &\leq \frac{1}{2\delta} \omega_1\left(\frac{\varepsilon}{\delta}\right) \int_{E \cup F} \abs{u(x)}^2 \,dx.
    \end{align*}
\end{proof}

\begin{proof}[Proof of Theorem \ref{thm:mod}]
    To simplify the presentation, we will suppress reference to the subscript $k$.
    If the right hand side of (\ref{eq:modification}) is infinite, then there is nothing to prove, so suppose that there exists some constant $C > 0$ such that for all $\varepsilon > 0$,
    \begin{equation}\label{Fbounds}
        F_\varepsilon(u, D) + F_\varepsilon(w, \Omega \setminus D_\delta) \leq C.
    \end{equation}
    Therefore, by (\ref{eq:J}), we get the kinetic energy bounds
    \begin{equation}\label{kineticbounds}
        \JJ_\varepsilon(u, D \setminus D_\delta, D) + \JJ_\varepsilon(w, D \setminus D_\delta, \Omega \setminus D_\delta) \leq C.
    \end{equation}
    Choose $\sigma > 0$, and let $\tilde{\delta} \defeq \frac{\delta}{M}$ for some large $M$ depending only on $\sigma$.
    Partition $D \setminus D_\delta$ into $M$ shells $D_{j\tilde{\delta}} \setminus D_{(j+1)\tilde{\delta}}$, and rewrite the term above into a sum over the shells:
    \begin{align*}
        \JJ_\varepsilon(&u, D \setminus D_\delta, D) + \JJ_\varepsilon(w, D \setminus D_\delta, \Omega \setminus D_\delta) \\
            &= \sum_{j = 0}^{M-1} \left(\JJ_\varepsilon(u, D_{j\tilde{\delta}} \setminus D_{(j+1)\tilde{\delta}}, D) + \JJ_\varepsilon(w, D_{j\tilde{\delta}} \setminus D_{(j+1)\tilde{\delta}}, \Omega \setminus D_\delta)\right).
    \end{align*}
    Therefore, provided we choose $M \geq C/\sigma$, by (\ref{kineticbounds}) there exists a distinguished $0 \leq j \leq M-1$ such that, denoting $\tilde{D} \defeq D_{j\tilde{\delta}}$, we have
    \begin{equation}\label{shell1}
        \JJ_\varepsilon(u, \tilde{D} \setminus \tilde{D}_{\tilde{\delta}}, D) + \JJ_\varepsilon(w, \tilde{D} \setminus \tilde{D}_{\tilde{\delta}}, \Omega \setminus D_\delta) \leq \sigma.
    \end{equation}

    Within the shell $\tilde{D} \setminus \tilde{D}_{\tilde{\delta}}$, we will further consider shells of width $\varepsilon \ll \tilde{\delta}$.
    For $0 \leq i \leq K-1$ with $K$ the integer part of $\tilde{\delta}/(2\varepsilon)$, define the shells
    \begin{equation}
        A_i \defeq \set{x \in \tilde{D} \given i\varepsilon < d(x, \partial \tilde{D}) < (i+1)\varepsilon}. \label{eq:Ai}
    \end{equation}
    Furthermore, denote $d_i(x) \defeq d(x, \partial \tilde{D}_{i\varepsilon})$.

    Consider the sum
    \begin{equation}\label{eq:innershellsum}
        \sum_{i=0}^{K-1} \int_{\tilde{D}_{i\varepsilon} \setminus \tilde{D}_{\tilde{\delta}}} \abs{u-w}\min\set*{\omega_1(1), \omega_1\left(\frac{\varepsilon}{d_i(x)}\right) \frac{\varepsilon}{d_i(x)}} \,dx,
    \end{equation}
    which is bounded above by the integral
    \begin{equation}\label{eq:innershellintegral}
        \int_{\tilde{D} \setminus \tilde{D}_{\tilde{\delta}}} \abs{u-w} \sum_{i=0}^{K-1} \min\set*{\omega_1(1), \omega_1\left(\frac{\varepsilon}{d_i(x)}\right) \frac{\varepsilon}{d_i(x)}} \,dx.
    \end{equation}
    We observe that the integral in (\ref{eq:innershellsum}) is not taken over the slice $A_i$, but rather over $\tilde{D}_{i\varepsilon} \setminus \tilde{D}_{\tilde{\delta}} = A_i \cup \left(\tilde{D}_{(i+1)\varepsilon} \setminus \tilde{D}_{\tilde{\delta}}\right)$.
    Since $t \mapsto \omega_1(t) t$ is increasing, the minimum is determined by whether $\varepsilon \leq d_i(x)$ or $d_i(x) \leq \varepsilon$.

    Fix $x \in \tilde{D} \setminus \tilde{D}_{\tilde{\delta}}$, and let $i_x \in \NN \cap [0, 2K]$ indicate the unique slice $A_{i_x}$ such that $x \in A_{i_x}$.
    If $\abs{i - i_x} \geq 2$, then $d_i(x) \geq \varepsilon(\abs{i-i_x} - 1)$.
    Using again the fact that $t \mapsto \omega_1(t)t$ is increasing, we estimate
    \begin{align*}
        \sum_{i=0}^{K-1} &\min\set*{\omega_1(1), \omega_1\left(\frac{\varepsilon}{d_i(x)}\right) \frac{\varepsilon}{d_i(x)}} \\
            &\leq \sum_{\substack{i = 0 \\ \abs{i - i_x} \leq 1}}^{K-1} \omega_1(1)
                + \sum_{\substack{i = 0 \\ \abs{i - i_x} \geq 2}}^{K-1} \omega_1\left(\frac{\varepsilon}{d_i(x)}\right)\frac{\varepsilon}{d_i(x)} \\
            &\leq 3\omega_1(1) + \sum_{\substack{i = 0 \\ \abs{i - i_x} \geq 2}}^{K-1} \omega_1\left(\frac{1}{\abs{i-i_x}}\right)\frac{1}{\abs{i-i_x}} \\
            &\leq 3\omega_1(1) + 2\sum_{n=1}^\infty \omega_1\left(\frac{1}{n}\right)\frac{1}{n}.
    \end{align*}
    By hypothesis (\ref{H2*}), this is finite, so the integral (\ref{eq:innershellintegral}) is bounded above by a universal constant times $\norm{u - w}_{L^1(D \setminus D_\delta)}$.
    Since $u_k - w_k \to 0$ in $L^1(D \setminus D_\delta)$ as $k \to \infty$, this can be made as small as desired in the limit, so in particular we choose it small enough so that
    \[
        \sum_{i=0}^{K-1} \int_{\tilde{D}_{i\varepsilon} \setminus \tilde{D}_{\tilde{\delta}}} \abs{u-w}\min\set*{\omega_1(1), \omega_1\left(\frac{\varepsilon}{d_i(x)}\right) \frac{\varepsilon}{d_i(x)}} \,dx
            \leq \frac{\tilde{\delta}\sigma}{2}.
    \]

    Since $K$ is of the order $\frac{\tilde{\delta}}{2\varepsilon}$, there exists $0 \leq i \leq K-1$ such that
    \begin{equation}\label{shell2}
        \begin{split}
            \int_{A_i} \abs{u - w} &\omega_1(1) \,dx + \int_{\tilde{D}_{(i+1)\varepsilon} \setminus \tilde{D}_{\tilde{\delta}}} \abs{u - w} \omega_1\left(\frac{\varepsilon}{d_i(x)}\right)\frac{\varepsilon}{d_i(x)} \,dx \\
                &= \int_{\tilde{D}_{i\varepsilon} \setminus \tilde{D}_{\tilde{\delta}}} \abs{u-w}\min\set*{\omega_1(1), \omega_1\left(\frac{\varepsilon}{d_i(x)}\right) \frac{\varepsilon}{d_i(x)}} \,dx
                \leq \sigma \varepsilon.
        \end{split}
    \end{equation}

    Now, partition $\Omega$ into 4 disjoint sets:
    \begin{align*}
        P \defeq \tilde{D}_{\tilde{\delta}}, \qquad
        Q \defeq \tilde{D}_{(i+1)\varepsilon} \setminus \tilde{D}_{\tilde{\delta}}, \qquad
        R \defeq A_i, \qquad
        S \defeq \Omega \setminus \tilde{D}_{i\varepsilon}.
    \end{align*}
    Define $v \defeq \varphi u + (1-\varphi) w$, where $\varphi$ is a smooth cutoff function with $\varphi \equiv 1$ on $P \cup Q$, $\varphi \equiv 0$ on $S$, and $\norm{\nabla \varphi}_\infty \leq \frac{3}{\varepsilon}$ (by (\ref{eq:Ai}) the width of $R$ is of order $\varepsilon$).
    If we rearrange the nonlocal interactions appropriately, by (\ref{eq:J}) we can bound
    \begin{equation}\label{eq:nonlocal}
        \begin{split}
            \JJ_\varepsilon(v, \Omega)
                &\leq \JJ_\varepsilon(u, D) + \JJ_\varepsilon(w, \Omega \setminus D_\delta) + \JJ_\varepsilon(v, R) \\
                &\qquad +\JJ_\varepsilon(v, P, R \cup S) + \JJ_\varepsilon(v, Q, R \cup S) + 2\JJ_\varepsilon(v, R, S),
        \end{split}
    \end{equation}
    so inequality (\ref{eq:modification}) is established if we show that each of the last four terms goes to zero in the limit.

    First, we bound $\JJ_\varepsilon(v, R)$.
    Using convexity of the map $z \mapsto \abs{z}^2$, we have
    \[
        \abs{v(y) - v(x)}^2
            \lesssim \abs{u(y) - u(x)}^2 + \abs{w(y) - w(x)}^2 + \abs{u(x) - w(x)}^2 \abs{\varphi(y) - \varphi(x)}^2,
    \]
    where the notation $A \lesssim B$ means that there is some universal constant $C$ such that $A \leq CB$.
   Therefore, we can bound
    \begin{align*}
        \JJ_\varepsilon(v, R)
            &\lesssim \JJ_\varepsilon(u, R) + \JJ_\varepsilon(w, R) \\
            &\qquad + \frac{1}{4\varepsilon} \int_R \int_R J_\varepsilon(y-x)\abs{\varphi(y) - \varphi(x)}^2 \,dy \abs{u(x) - w(x)}^2 \,dx.
    \end{align*}
    By (\ref{shell1}), the first two terms are $\leq \sigma$.
    Since $\abs{u-w} \leq 2$ and $\abs{\varphi(y) - \varphi(x)} \leq 3(1 \wedge \frac{1}{\varepsilon}\abs{y - x})$, by (\ref{shell2}) and Lemma \ref{lemma:interior}, we obtain
    \begin{align*}
        \frac{1}{4\varepsilon} \int_R \int_R J_\varepsilon(y-x) \abs{\varphi(y) &- \varphi(x)}^2 \,dy \abs{u(x) - w(x)}^2 \,dx \\
            &\leq \frac{9M_J}{2\varepsilon} \int_R \abs{u-w} \,dx \\
            &\leq \frac{9M_J}{2\omega_1(1)} \sigma.
    \end{align*}
    Therefore, $\JJ_\varepsilon(v, R) \lesssim \sigma$.

    Next, we bound $\JJ_\varepsilon(v, P, R \cup S)$.
    Since $\abs{v(y) - v(x)} \leq 2$ and $d(P, R \cup S) \geq \tilde{\delta}/2$, we can bound
    \begin{align*}
        \JJ_\varepsilon(v, P, R \cup S)
            &= \frac{1}{4\varepsilon} \int_P \int_{R \cup S} J_\varepsilon(y-x) \abs{v(y) - v(x)}^2 \,dy\, dx \\
            &\leq \frac{1}{\varepsilon} \int_P \int_{B_{\tilde{\delta}/2}(x)^c} J_\varepsilon(y-x) \,dy\,dx \\
            &\leq \frac{1}{\varepsilon} \int_\Omega \int_{B_{\tilde{\delta}/(2\varepsilon)}^c} J(h) \,dh\,dx \\
            &\leq \frac{2\abs{\Omega}}{\tilde{\delta}} \omega_1\left(\frac{2\varepsilon}{\tilde{\delta}}\right).
    \end{align*}

    We will estimate $\JJ_\varepsilon(v, Q, R \cup S)$ in two steps.
    If $x \in Q, y \in S$, then $\abs{y-x} \geq d_i(x)$ and $\abs{v(y) - v(x)}^2 \leq 2\abs{w(y) - w(x)}^2 + 2\abs{u(x) - w(x)}^2$.
    Therefore, using (\ref{eq:omega1}), (\ref{shell1}), and (\ref{shell2}) we get
    \begin{align*}
        \JJ_\varepsilon(v, Q, S)
            &\leq 2\JJ_\varepsilon(w, Q, S) + \frac{1}{2\varepsilon} \int_Q \int_S J_\varepsilon(y-x) \,dy\, \abs{u(x) - w(x)}^2 \,dx \\
            &\leq 2\JJ_\varepsilon(w, Q, S) + \frac{1}{\varepsilon} \int_Q \abs{u(x) - w(x)} \omega_1\left(\frac{\varepsilon}{d_i(x)}\right) \frac{\varepsilon}{d_i(x)} \,dx \\
            &\lesssim \sigma.
    \end{align*}

    On the other hand, if $x \in R, y \in Q$, then $\abs{y-x} \geq d_{i+1}(x)$ and $\abs{1 - \varphi(x)} \leq \frac{3}{\varepsilon}d_{i+1}(x)$ by the gradient bounds on $\varphi$.
    Furthermore, by the convexity of $z \mapsto \abs{z}^2$, also $\abs{v(y) - v(x)}^2 \leq 2\abs{u(y) - u(x)}^2 + 2\abs{1 - \varphi(x)}^2 \abs{u(x) - w(x)}^2$.
    We thus have
    \begin{align*}
        \JJ_\varepsilon(v, R, Q)
            &\leq 2\JJ_\varepsilon(u, R, Q) \\
            &\quad\qquad+ \frac{1}{2\varepsilon} \int_R \int_Q J_\varepsilon(y-x) \,dy \abs{1 - \varphi(x)}^2 \abs{u(x) - w(x)}^2 \,dx \\
            &\leq 2\JJ_\varepsilon(u, R, Q) \\
            &\quad\qquad+ \frac{9}{\varepsilon} \int_R \omega_1\left(\frac{\varepsilon}{d_{i+1}(x)}\right) \frac{d_{i+1}(x)}{\varepsilon} \abs{u(x) - w(x)} \,dx.
    \end{align*}
     Note that $\frac{\varepsilon}{d_{i+1}(x)} \geq 1$ for all $x \in R$.
     Since $\omega_1(t)/t \leq M_J$ for all $t \geq 1$, we have
     \[
         \omega_1\left(\frac{\varepsilon}{d_{i+1}(x)}\right) \frac{d_{i+1}(x)}{\varepsilon} \leq M_J,
     \]
     for all $x \in R$.
     Taking this into account, and again using (\ref{shell1}) and (\ref{shell2}), we deduce that
     \[
         \JJ_\varepsilon(v, R, Q) \lesssim \sigma.
     \]
     Therefore, combining these results, we get
     \[
         \JJ_\varepsilon(v, Q, R \cup S) \lesssim \sigma.
     \]

     Lastly, we bound $\JJ_\varepsilon(v, R, S)$.
     If $x \in R, y \in S$, then $\abs{\varphi(x)} \leq \frac{3}{\varepsilon} d_i(x)$ and $\abs{v(y) - v(x)}^2 \leq 2\abs{u(y) - u(x)}^2 + 2\abs{\varphi(x)}^2 \abs{u(x) - w(x)}^2$.
     Therefore, using the same argument as that to bound $\JJ_\varepsilon(v, R, Q)$, we get
     \[
         \JJ_\varepsilon(v, R, S) \lesssim \sigma.
     \]

     Substituting these bounds into (\ref{eq:nonlocal}), this gives (for arbitrary $\sigma > 0$)
     \begin{equation}
         \limsup_{k \to \infty} \JJ_{\varepsilon_k}(v_k, \Omega)
            \leq \limsup_{k \to \infty} \left(\JJ_{\varepsilon_k}(u_k, D) + \JJ_{\varepsilon_k}(w_k, \Omega \setminus D_\delta)\right) + C\sigma,
    \end{equation}
    where $C$ is a constant depending only on $M_J$.

    All that remains is to bound the potential term $\WW_\varepsilon(v, \Omega)$, but for this we can follow exactly the same argument as \cite{savin_-convergence_2012}.
    Letting $\sigma \to 0^+$ completes the proof.
\end{proof}

\section{$\Gamma$-Limsup Inequality}

The proof of Theorem \ref{item:limsup} is a simple consequence of the following theorem.

\begin{theorem}\label{thm:limsup}
    For every $u \in BV(\Omega; \set{-1, 1})$, there exists a sequence $u_j \to u$ in $L^p(\Omega)$ such that
    \[
        \limsup_{j \to \infty} F_{\varepsilon_j}(u_j, \Omega) \leq \int_{S_u} \psi(\nu_u) \,d\Hausdorff^{N-1}.
    \]
\end{theorem}

\begin{proof}[Proof of Theorem \ref{item:limsup} supposing Theorem \ref{thm:limsup}]
    Let $\varepsilon_j \to 0$ and $u \in L^1(\Omega)$.
    If $u \not\in BV(\Omega; \set{-1, 1})$, then $F(u) = +\infty$ and there is nothing to prove, so suppose $u \in BV(\Omega; \set{-1, 1})$.
    Now simply apply Theorem \ref{thm:limsup}.
\end{proof}

We will first prove Theorem \ref{thm:limsup} for BV functions with polyhedral jump sets, then address the general case.
In order to do so, we will need a fixed profile to glue across facets of the polyhedron, interpolating between optimal profiles in different directions.

Fix a mollifier $\theta \in C_c^\infty(\RR^N)$ such that $\supp \theta \subset B_1(0)$, $\int_{\RR^N} \theta \,dx = 1$, and $\norm{\nabla \theta}_\infty \leq 2$ and define $\theta_\tau(x) \defeq \tau^{-N}\theta(x/\tau)$ for all $\tau > 0$.
For any $u \in BV(\Omega; \set{-1, 1})$, define
\begin{equation}\label{moll}
    \tilde{u}_\varepsilon \defeq u * \theta_\varepsilon,
\end{equation}
to be our fixed profile of scale $\varepsilon$.
By construction, $\tilde{u}_\varepsilon(x) \in [-1, 1]$ for all $x \in \Omega$.

\begin{proposition}\label{prop:isotropic}
    Let $u \in BV_{loc}(\RR^N; \set{-1, 1})$, and for every $\varepsilon > 0$, let $\tilde{u}_\varepsilon$ be defined as in (\ref{moll}).
    Assume there exists a bounded polyhedral set $\Sigma$ of dimension $N-1$ such that $S_u = \Sigma$, let $\Sigma^{N-2}$ be the union of all its $(N-2)$-dimensional facets, and let $(\Sigma^{N-2})^\delta \defeq \set{x \in \RR^N \given d(x, \Sigma^{N-2}) < \delta}$.
    Then, there exist constants $\delta_\Sigma > 0, C = C(J, p, \Sigma) > 0$ such that for $0 < \varepsilon < \delta < \delta_\Sigma$, we have
    \[
        \JJ_\varepsilon(\tilde{u}_\varepsilon, (\Sigma^{N-2})^\delta)
            \leq C \delta \Hausdorff^{N-2}(\Sigma^{N-2}).
    \]
\end{proposition}

\begin{proof}[Proof of Proposition \ref{prop:isotropic}]
    Write $(\Sigma^{N-2})^\delta$ as the disjoint union of two sets,
    \begin{align*}
        C_\varepsilon^\delta
            &\defeq \set{x \in (\Sigma^{N-2})^\delta \given d(x, \Sigma^{N-2}) < \varepsilon}, \\
        \hat{C}_\varepsilon^\delta
            &\defeq \set{x \in (\Sigma^{N-2})^\delta \given d(x, \Sigma^{N-2}) \geq \varepsilon}.
    \end{align*}
    Note that the near field set $C_\varepsilon^\delta$ is nothing more than $(\Sigma^{N-2})^\varepsilon$.

    We can break up the kinetic energy into near and far field interactions on these sets as
    \[
        \JJ_\varepsilon(\tilde{u}_\varepsilon, (\Sigma^{N-2})^\delta)
            \leq 2\JJ_\varepsilon(\tilde{u}_\varepsilon, C_\varepsilon^\delta, (\Sigma^{N-2})^\delta)
            + 2\JJ_\varepsilon(\tilde{u}_\varepsilon, \hat{C}_\varepsilon^\delta).
    \]
    Using the inequality $0 \leq \abs{\tilde{u}_\varepsilon(y) - \tilde{u}_\varepsilon(x)}^2 \leq \left(\frac{2}{\varepsilon}\abs{y-x} \wedge 2\right)^2$, we can apply Lemma \ref{lemma:interior} to get
    \begin{align*}
        \JJ_\varepsilon(\tilde{u}_\varepsilon, C_\varepsilon^\delta, (\Sigma^{N-2})^\delta)
            &= \frac{1}{4\varepsilon} \int_{C_\varepsilon^\delta} \int_{(\Sigma^{N-2})^\delta} J_\varepsilon(y-x) \abs{\tilde{u}_\varepsilon(y) - \tilde{u}_\varepsilon(x)}^2 \,dy\,dx \\
            &\leq \frac{M_J}{\varepsilon} \LL^N(C_\varepsilon^\delta) \\
            &\lesssim M_J \varepsilon \Hausdorff^{N-2}(\Sigma^{N-2})
             \leq M_J \delta \Hausdorff^{N-2}(\Sigma^{N-2}),
    \end{align*}
    for $\varepsilon$ and $\delta$ sufficiently small.

    For the second term, we note that $\tilde{u}_\varepsilon \equiv \pm 1$ on $\hat{C}_\varepsilon^\delta$, and that the subsets $\set{\tilde{u}_\varepsilon = +1}$ and $\set{\tilde{u}_\varepsilon = -1}$ are separated by a gap of size $2\varepsilon$.
    Denote by $C_+^\delta$ and $C_-^\delta$ the parts of $\hat{C}_\varepsilon^\delta$ where $\tilde{u}_\varepsilon$ equals $+1$ or $-1$ respectively.
    Applying the Blaschke-Petkantschin formula (Lemma \ref{lem:abp}) to this term, we get
    \begin{equation}\label{eq:slicing}
        \begin{split}
        &\phantom{=~} \JJ_\varepsilon(\tilde{u}_\varepsilon, \hat{C}_\varepsilon^\delta)
            = \frac{2}{\varepsilon} \int_{C_+^\delta} \int_{C_-^\delta} J_\varepsilon(y-x) \,dy \,dx \\
            &= \frac{1}{\varepsilon} \int_{\Graff(N, 1)} \int_{L \cap C_+^\delta} \int_{L \cap C_-^\delta} J_\varepsilon(y-x) \abs{y-x}^{N-1} \,d\Hausdorff^1(y) \,d\Hausdorff^1(x) \,d\lambda^N_1(L) \\
            &= \frac{1}{\varepsilon} \int_{\Graff(N, 1)} \mathcal{A}(L) \,d\lambda_1^N(L) \\
            &= \frac{1}{2\varepsilon} \int_{S^{N-1}} \int_{\xi^\perp} \mathcal{A}(z + \RR\xi) \,d\Hausdorff^{N-1}(z) \,d\Hausdorff^{N-1}(\xi),
        \end{split}
    \end{equation}
    where we have defined $\mathcal{A} \colon \Graff(N,1) \to [0, \infty)$ by
    \begin{equation}\label{eq:AA}
        \mathcal{A}(L)
            \defeq \int_{L \cap C_+^\delta} \int_{L \cap C_-^\delta} J_\varepsilon(y-x)\abs{y-x}^{N-1} \,d\Hausdorff^1(y) \,d\Hausdorff^1(x).
    \end{equation}

    Since the integrand in (\ref{eq:AA}) only depends on terms of the form $y-x$, the integrand is independent of $z$.
    Furthermore, note that $\mathcal{A}(L) \neq 0$ only if $L$ intersects both $C_+^\delta$ and $C_-^\delta$; that is, if $z \in \hat{C}_\varepsilon^\delta \vert \xi^\perp$, where $U \vert H$ denotes the projection of $U$ onto the hyperplane $H$.
    Therefore,
    \begin{equation}\label{eq:Abounds}
        \begin{split}
            \mathcal{A}&(z + \RR\xi) \\
                &= \int_{(z + \RR\xi) \cap C_+^\delta} \int_{(z + \RR\xi) \cap C_-^\delta} J_\varepsilon(y-x) \abs{y-x}^{N-1} \,d\Hausdorff^1(x) \,d\Hausdorff^1(y) \\
                &= \int_{\RR\xi \cap (C_+^\delta - z)} \int_{\RR\xi \cap (C_-^\delta - z)} J_\varepsilon(y-x) \abs{y-x}^{N-1} \,d\Hausdorff^1(x) \,d\Hausdorff^1(y) \\
                &\leq \chi_{\hat{C}_\varepsilon^\delta \vert \xi^\perp}(z) \int_{-\infty}^{-\varepsilon} \int_\varepsilon^\infty J_\varepsilon\left((t-s)\xi\right) \abs{t-s}^{N-1} \,dt \,ds \\
                &= \varepsilon \chi_{\hat{C}_\varepsilon^\delta \vert \xi^\perp}(z) \int_{-\infty}^{-1} \int_1^\infty J^\xi(t-s) \,dt \,ds,
        \end{split}
    \end{equation}
    where we define $J^\xi(t) \defeq J(t\xi)\abs{t}^{N-1}$.

    Set $F(\xi) \defeq \int_{-\infty}^{-1} \int_1^\infty J^\xi(t-s) \,dt \,ds$.
    Substituting (\ref{eq:Abounds}) back into (\ref{eq:slicing}), we get the estimate
    \begin{equation}\label{eq:ff}
        \begin{split}
            \JJ_\varepsilon(\tilde{u}_\varepsilon, \hat{C}_\varepsilon^\delta)
                &\lesssim \int_{S^{N-1}} \int_{\xi^\perp} \chi_{\hat{C}_\varepsilon^\delta \vert \xi^\perp}(z) F(\xi) \,d\Hausdorff^{N-1}(z) \,d\Hausdorff^{N-1}(\xi) \\
                &= \int_{S^{N-1}} \Hausdorff^{N-1}(\hat{C}_\varepsilon^\delta \vert \xi^\perp) F(\xi) \,d\Hausdorff^{N-1}(\xi) \\
                &\leq \int_{S^{N-1}} \Hausdorff^{N-1}((\Sigma^{N-2})^\delta \vert \xi^\perp) F(\xi) \,d\Hausdorff^{N-1}(\xi).
        \end{split}
    \end{equation}
    Recognizing $(\Sigma^{N-2})^\delta \vert \xi^\perp = (\Sigma^{N-2} \vert \xi^\perp)^\delta \subset \xi^\perp$, we can apply Lemma \ref{lem:steiner} to $X = \Sigma^{N-2} \vert \xi^\perp$ in the hyperplane $\xi^\perp \cong \RR^{N-1}$ to get that for $\delta \ll 1$,
    \[
        \Hausdorff^{N-1}((\Sigma^{N-2})^\delta \vert \xi^\perp) \lesssim \delta \Hausdorff^{N-2}(\Sigma^{N-2} \vert \xi^\perp) \leq \delta \Hausdorff^{N-2}(\Sigma^{N-2}).
    \]
    Substituting this into (\ref{eq:ff}), we find
    \[
        \JJ_\varepsilon(\tilde{u}_\varepsilon, \hat{C}_\varepsilon^\delta) \lesssim \delta \Hausdorff^{N-2}(\Sigma^{N-2}) \int_{S^{N-1}} F(\xi) \,d\Hausdorff^{N-1}(\xi).
    \]

    To conclude the bound of the far field terms, it only remains to show that $F \in L^1(S^{N-1})$, but this is a consequence of hypothesis (\ref{H2*}).
    To be precise,
    \begin{align*}
        \int_{S^{N-1}} F(\xi) \,d\Hausdorff^{N-1}(\xi)
            &= \int_{S^{N-1}} \int_{-\infty}^{-1} \int_1^\infty J^\xi(t-s) \,dt \,ds \,d\Hausdorff^{N-1}(\xi) \\
            &= \int_{S^{N-1}} \int_2^\infty \int_s^\infty J^\xi(t) \,dt \,ds \,d\Hausdorff^{N-1}(\xi) \\
            &\leq \int_{S^{N-1}} \int_2^\infty \int_s^\infty \frac{t}{s} J^\xi(t) \,dt \,ds \,d\Hausdorff^{N-1}(\xi) \\
            &= \int_{S^{N-1}} \int_2^\infty t J^\xi(t) \int_2^t \frac{1}{s} \,ds \,dt \,d\Hausdorff^{N-1}(\xi) \\
            &\leq \int_{S^{N-1}} \int_2^\infty t J^\xi(t) \ln t \,dt \,d\Hausdorff^{N-1}(\xi) \\
            &= \int_{B_2^c} J(h) \abs{h} \ln\abs{h} \,dh < +\infty.
    \end{align*}

    Combining the near field bounds and far field bounds, we conclude.
\end{proof}

The proof of the $\Gamma$-limsup requires the following two lemmas, from which Theorem \ref{thm:limsup} will follow.
The proof of Lemma \ref{lem:poly} is identical to step 2 in the proof of Theorem 5.2 in \cite{alberti_non-local_1998}, so will be omitted (see also \cite[Lemma 7.3]{dal_maso_asymptotic_2017}).
\begin{lemma}\label{lem:poly}
    Let $P$ be a bounded polyhedron of dimension $N-1$ containing 0 with normal $\nu$, let $\rho > 0$, and let $P_\rho$ be the $N$-dimensional prism $\set{x + t\nu \given x \in P, t \in (-\frac{\rho}{2}, \frac{\rho}{2})}$.
    Define $w^\nu \colon P^\rho \to \RR$ by
    \[
        w^\nu(x) = \begin{cases}
            +1 \textup{ if } x \cdot \nu > 0, \\
            -1 \textup{ if } x \cdot \nu \leq 0.
        \end{cases}
    \]
    Then, for every $\eta > 0$ there exists a sequence $\set{u_\varepsilon} \subset L^p(P_\rho; [-1, 1])$ such that $u_\varepsilon \to w^\nu$ in $L^p(P_\rho)$ and
    \[
        \limsup_{\varepsilon \to 0} F_\varepsilon(u_\varepsilon, P_\rho)
            \leq \Hausdorff^{N-1}(P)\left(\psi(\nu) + \eta\right).
    \]
\end{lemma}

\begin{lemma}\label{lem:ls}
    Let $u \in BV_{loc}(\RR^N; \set{-1, 1})$ be such that there exists a polyhedral set $\Sigma$ of dimension $N-1$ such that $S_u = \Sigma$.
    For every $\sigma > 0$ there exist $\rho > 0$ and $\delta \in (0, \rho)$ with the following property: for every $\varepsilon_j \to 0$ there exists $v_j \in L^p(\Sigma^\rho; [-1, 1])$ such that $v_j = u$ on $\Sigma^\rho \setminus \Sigma^{\rho - \delta}$ and
    \[
        \limsup_{j \to \infty} F_{\varepsilon_j}(v_j, \Sigma^\rho)
            \leq \int_\Sigma \psi(\nu_u) \,d\Hausdorff^{N-1} + \sigma.
    \]
\end{lemma}

\begin{proof}
    The proof of Lemma \ref{lem:ls} is very similar to the proof of Lemma 7.4 in \cite{dal_maso_asymptotic_2017}, but with some simplifications made to adapt to our functional.

    Let $\delta_\Sigma > 0$ be as in Proposition \ref{prop:isotropic}, and fix $\sigma, \hat{\sigma} > 0$ with $\hat{\sigma} < \min\set{\delta_\Sigma, \sigma}$.
    There exists $\rho \in (0, \hat{\sigma})$ and a finite number of bounded polyhedra $P^1, \dotsc, P^k$ of dimension $N-1$ with normals $\nu^1, \dotsc, \nu^k$ contained in the facets of $\Sigma$ such that the closed prisms $\overline{P}^i_\rho$ are pairwise disjoint and
    \begin{equation}\label{prisms}
       \Sigma_\rho \setminus \bigcup_{i=1}^k P^i_\rho \subset (\Sigma^{N-2})^{\hat{\sigma}}.
    \end{equation}
    Find $R^1, \dotsc, R^k$ bounded polyhedra of dimension $N-1$ also contained in the facets of $\Sigma$ such that the closed prisms are disjoint and $P^i \subset\subset R^i$ for each $i$.

    Fix $\eta > 0$ such that $\eta \Hausdorff^{N-1}(\Sigma) < \sigma/2$, so that by Lemma \ref{lem:poly} for each polyhedron $R^i$ there exists a sequence $\set{u^i_j} \subset L^p(R^i_\rho; [-1, 1])$ such that $u^i_j \to u$ in $L^p(R^i_\rho)$, and
    \[
        \limsup_{j \to \infty} F_{\varepsilon_j}(u^i_j, R^i_\rho)
            \leq \Hausdorff^{N-1}(R^i)(\psi(\nu^i) + \eta).
    \]
    Choose $0 < \delta < \min\set{\rho/2, \hat{\sigma}}$ and, as in (\ref{moll}), define $\tilde{u}_j \defeq u * \theta_{\varepsilon_j}$.
    We apply the modification theorem, Theorem \ref{thm:mod}, with $\Omega = R^i_\rho$ and $D \defeq (R^i_\rho)_\delta$ to glue together $u^i_j$ and $\tilde{u}_j$ to create functions $v^i_j \in L^2(R^i_\rho)$ satisfying
    \begin{enumerate}
        \item $v^i_j \equiv u^i_j$ in $(R^i_\rho)_{2\delta}$,
        \item $v^i_j \equiv \tilde{u}_j$ on $R^i_\rho \setminus (R^i_\rho)_{2\delta}$,
        \item and
            \begin{equation}\label{eq:modshell}
                \limsup_{j \to \infty} F_{\varepsilon_j}(v^i_j, R^i_\rho)
                    \leq \limsup_{j \to \infty} \left(F_{\varepsilon_j}(u^i_j, (R^i_\rho)_\delta) + F_{\varepsilon_j}(\tilde{u}_j, R^i_\rho \setminus (R^i_\rho)_{2\delta})\right).
            \end{equation}
    \end{enumerate}

    We now have to bound $F_{\varepsilon_j}(\tilde{u}_j, R^i_\rho \setminus (R^i_\rho)_{2\delta})$.
    As in the proof of Proposition \ref{prop:isotropic}, to bound $\JJ_{\varepsilon_j}(\tilde{u}_j, R^i_\rho \setminus (R^i_\rho)_{2\delta})$, we break up the shell into a near field and far field;
    \begin{align*}
        C_j &\defeq \set{x \in R^i_\rho \setminus (R^i_\rho)_{2\delta} \given d(x, R^i) < \varepsilon_j}, \\
        \hat{C}_j &\defeq \set{x \in R^i_\rho \setminus (R^i_\rho)_{2\delta} \given d(x, R^i) \geq \varepsilon_j}.
    \end{align*}
    Therefore,
    \[
        \JJ_{\varepsilon_j}(\tilde{u}_j, R^i_\rho \setminus (R^i_\rho)_{2\delta})
            \leq 2\JJ_{\varepsilon_j}(\tilde{u}_j, C_j, R^i_\rho \setminus (R^i_\rho)_{2\delta})
               + 2\JJ_{\varepsilon_j}(\tilde{u}_j, \hat{C}_j),
    \]
    and, as in the proof of Theorem \ref{thm:mod}, again the first term is $\lesssim \delta C_{R^i_\rho}$ by Lemma \ref{lemma:interior}, where $C_{R^i_\rho}$ is a constant depending on $R^i_\rho$.
    We break up $\hat{C}_j$ into the following four sets
    \begin{align*}
        C_j^+ &\defeq \set{x \in \hat{C}_j \given d(x, R^i) \geq \rho - 2\delta, u(x) = 1}, \\
        S_j^+ &\defeq \set{x \in \hat{C}_j \given d(x, R^i) < \rho - 2\delta, u(x) = 1}, \\
        S_j^- &\defeq \set{x \in \hat{C}_j \given d(x, R^i) < \rho - 2\delta, u(x) = -1}, \\
        C_j^- &\defeq \set{x \in \hat{C}_j \given d(x, R^i) \geq \rho - 2\delta, u(x) = -1},
    \end{align*}
    so that by additivity of $\JJ$,
    \[
        \JJ_{\varepsilon_j}(\tilde{u}_j, \hat{C}_j)
            \leq \JJ_{\varepsilon_j}(\tilde{u}_j, C_j^+, S_j^- \cup C_j^-)
               + \JJ_{\varepsilon_j}(\tilde{u}_j, S_j^+, S_j^-)
               + \JJ_{\varepsilon_j}(\tilde{u}_j, C_j^+ \cup S_j^+, C_j^-).
    \]
    The first and third terms involve sets at distance $\geq \rho - 2\delta + \varepsilon$, so by Lemma \ref{lemma:separation} they will tend to zero as $j \to \infty$.
    By (\ref{prisms}), the sets $S_j^\pm$ are actually both subsets of $(\Sigma^{N-2})^{\hat{\sigma} + 2\delta} \subset (\Sigma^{N-2})^{3\hat{\sigma}}$, so by Proposition \ref{prop:isotropic}, the second term is $\lesssim \hat{\sigma}$; therefore, we get
    \begin{equation}\label{eq:JJshell}
        \limsup_{j \to \infty} \JJ_{\varepsilon_j}(\tilde{u}_j, R^i_\rho \setminus (R^i_\rho)_{2\delta}) \lesssim \hat{\sigma}.
    \end{equation}

    We can also bound the potential energy by noting that $W(u) = 0$ on $\hat{C}_j$, and so
    \begin{equation}\label{eq:WWshell}
        \limsup_{j \to \infty} \WW_{\varepsilon_j}(\tilde{u}_j, R^i_\rho \setminus (R^i_\rho)_{2\delta})
            = \limsup_{j \to \infty} \WW_{\varepsilon_j}(\tilde{u}_j, C_j)
            \leq M_W C_{R^i_\rho} \delta,
    \end{equation}
    where $M_W \defeq \max_{[-1,1]} W$.
    Combining inequalities (\ref{eq:JJshell}) and (\ref{eq:WWshell}), we have that
    \[
        \limsup_{j \to \infty} F_{\varepsilon_j}(\tilde{u}_j, \hat{C}_j)
            \leq \kappa_1 \hat{\sigma},
    \]
    for a constant $\kappa_1$ independent of $u$.
    This, together with (\ref{eq:modshell}) and by Lemma \ref{lem:poly}, yields
    \begin{equation}\begin{aligned}\label{eq:prismbound}
        \limsup_{j \to \infty} F_{\varepsilon_j}(v^i_j, R^i_\rho)
            &\leq \limsup_{j \to \infty} F_{\varepsilon_j}(u^i_j, R^i_\rho) + \kappa_1 \hat{\sigma} \\
            &\leq \Hausdorff^{N-1}(R^i)(\psi(\nu^i) + \eta) + \kappa_1 \hat{\sigma}.
    \end{aligned}\end{equation}

    Define $v_j \defeq v^i_j$ in $R^i_\rho$, and $v_j \defeq \tilde{u}_j$ on $A_\rho \defeq \Sigma^\rho \setminus \bigcup_{i=1}^k R^i_\rho$.
    Then $v_j \to u$ in $L^p(\Sigma^\rho)$ and $v_j = u$ on $\Sigma^\rho \setminus \Sigma^{\rho - \delta}$ for $j \gg 1$.

    By additivity of $\WW$, we obtain
    \begin{equation}\label{eq:WWadd}
        \WW_{\varepsilon_j}(v_j, \Sigma^\rho)
            \leq \sum_{i=1}^k \WW_{\varepsilon_j}(v_j, R^i_\rho) + \WW_{\varepsilon_j}(v_j, A_\rho),
    \end{equation}
    but $v_j = \tilde{u}_j$ on $A_\rho$, and $\tilde{u}_j = \pm 1$ for $x \not\in \Sigma^{2\varepsilon_j}$, so we can bound
    \begin{align*}
        \WW_{\varepsilon_j}(v_j, A_\rho)
            &\leq \WW_{\varepsilon_j}(\tilde{u}_j, (\Sigma^{N-2})^{\hat{\sigma}} \cap \Sigma^{2\varepsilon_j}) \\
            &\leq \frac{1}{\varepsilon_j} M_W \LL^N\left((\Sigma^{N-2})^{\hat{\sigma}} \cap \Sigma^{2\varepsilon_j}\right) \\
            &\leq M_W c_\Sigma \hat{\sigma} \Hausdorff^{N-2}(\Sigma^{N-2}),
    \end{align*}
    where $c_\Sigma > 0$ is a constant depending only on $\Sigma$.
    Combining the above inequality with (\ref{eq:WWadd}), we get
    \begin{equation}\label{eq:Wbound}
        \WW_{\varepsilon_j}(v_j, \Sigma^\rho)
            \leq \sum_{i=1}^k \WW_{\varepsilon_j}(v_j, R^i_\rho) + M_W c_\Sigma \hat{\sigma} \Hausdorff^{N-2}(\Sigma^{N-2}).
    \end{equation}

    To bound the kinetic energy, we use additivity to get
    \begin{equation}\begin{aligned}\label{eq:JJadd}
        \JJ_{\varepsilon_j}(v_j, \Sigma^\rho)
            &\leq \sum_{i=1}^k \JJ_{\varepsilon_j}(v_j, R^i_\rho) + 2\sum_{i=1}^k \JJ_{\varepsilon_j}(v_j, P^i_\rho, \Sigma^\rho \setminus R^i_\rho) \\
            &\quad+ \JJ_{\varepsilon_j}(v_j, (\Sigma^{N-2})^{\hat{\sigma}}) + \sum_{\substack{i = 1 \\ i \neq j}}^k \JJ_{\varepsilon_j}(v_j, R^i_\rho, R^j_\rho).
    \end{aligned}\end{equation}

    The sets in the second and fourth terms are at a fixed distance apart, so the corresponding terms tend to zero in the limit as $j \to \infty$.
    By Lemma \ref{prop:isotropic},
    \[
        \JJ_{\varepsilon_j}(v_j, (\Sigma^{N-2})^{\hat{\sigma}})
            \leq \kappa \hat{\sigma} \Hausdorff^{N-2}(\Sigma^{N-2}),
    \]
    where $\kappa$ does not depend on $v_j$.
    Combining this inequality with (\ref{eq:Wbound}), (\ref{eq:prismbound}), and (\ref{eq:JJadd}), we get
    \begin{align*}
        \limsup_{j \to \infty}F_{\varepsilon_j}(v_j, \Sigma^\rho)
                &\leq \int_\Sigma \psi(\nu_u) d\Hausdorff^{N-1} + \eta\Hausdorff^{N-1}(\Sigma) \\
                &\quad+ \kappa_1\hat{\sigma} + M_W c_\Sigma \hat{\sigma}\Hausdorff^{N-2}(\Sigma^{N-2}) + \kappa \hat{\sigma}\Hausdorff^{N-2}(\Sigma^{N-2}).
    \end{align*}
    Since $\eta\Hausdorff^{N-1}(\Sigma) < \sigma/2$, we can conclude by taking $\hat{\sigma}$ sufficiently small.
\end{proof}

From here, the proof of Theorem \ref{thm:limsup} follows from Lemma \ref{lem:ls}, approximating a general $u \in BV(\RR^N; \set{-1,1})$ by functions with polyhedral jump set (cf. \cite[Section 5.4]{alberti_non-local_1998}, \cite[Theorem 7.1]{dal_maso_asymptotic_2017}) and diagonalizing.
We sketch the proof below.

\begin{proof}[Proof of Theorem \ref{thm:limsup}]
    By \cite[Lemma 3.1]{baldo_minimal_1990}, for every $u \in BV(\Omega; \set{-1, 1})$ there exists a sequence $\set{u_k}$ in $BV(\Omega; \set{-1, 1})$ converging to $u$ in $L^1(\Omega)$ such that each jump set $S_{u_k}$ is the intersection of $\Omega$ and a polyhedral set of dimension $N-1$.
    Furthermore, since $\psi$ is upper semi-continuous, by Reshetnyak's convergence theorem (see \cite{spector_simple_2011}), we have that
    \[
        \limsup_{k \to \infty} \int_{S_{u_k}} \psi(\nu_{u_k}) \,d\Hausdorff^{N-1}
            \leq \int_{S_u} \psi(\nu_u) \,d\Hausdorff^{N-1}.
    \]
    Therefore, it is sufficient to prove the theorem for functions with polyhedral jump set.

    Let $u \in BV(\Omega; \set{-1, 1})$ be such that $S_u = \Omega \cap \Sigma$ with $\Sigma$ a bounded polyhedral set of dimension $N-1$.
    For all $\sigma > 0$, let $0 < \delta < \rho$ and $v_j$ be as in Lemma \ref{lem:ls}.
    Define $u_j \in L^2(\Omega)$ by $u_j \defeq v_j$ in $\Sigma^\rho$ and $u_j \defeq u$ in $\Omega \setminus \Sigma^\rho$.
    Since $v_j = u$ on $\Sigma^\rho \setminus \Sigma^{\rho - \delta}$ for $j \gg 1$, we have that $u_j = u$ on $\Omega \setminus \Sigma^{\rho - \delta}$.
    Therefore,
    \begin{equation}\label{eq:ww}
        \WW_{\varepsilon_j}(u_j, \Omega) \leq \WW_{\varepsilon_j}(v_j, \Sigma^\rho),
    \end{equation}
    for sufficiently large $j$.

    On the other hand, we can estimate $\JJ_{\varepsilon_j}(u_j, \Omega)$ by
    \begin{equation}\label{eq:jj}
        \JJ_{\varepsilon_j}(u_j, \Omega)
            \leq \JJ_{\varepsilon_j}(v_j, \Sigma^\rho) + 2\JJ_{\varepsilon_j}(u_j, \Sigma^{\rho - \delta}, \Omega \setminus \Sigma^\rho) + \JJ_{\varepsilon_j}(u, \Omega \setminus \Sigma^{\rho - \delta}).
    \end{equation}
    Since $d(\Sigma^{\rho-\delta}, \Omega \setminus \Sigma^\rho) > \delta$, we can apply Lemma \ref{lemma:separation} to the second term to get the inequality
    \[
        \JJ_{\varepsilon_j}(u_j, \Sigma^{\rho - \delta}, \Omega \setminus \Sigma^\rho)
            \leq \frac{1}{2\delta} \omega_1\left(\frac{\varepsilon_j}{\delta}\right) \int_\Omega \abs{u_j(x)}^2 \,dx.
    \]
    Therefore, as $j \to \infty$, this term vanishes from (\ref{eq:jj}).
    Similarly, the third term tends to zero by Lemma \ref{lemma:separation} since $\Omega \setminus \Sigma^{\rho - \delta}$ can be partitioned into two sets, on one of which $u \equiv +1$ and on the other $u \equiv -1$, separated by a gap of width $2(\rho - \delta)$.

    Therefore, by (\ref{eq:ww}), (\ref{eq:jj}), and Lemma \ref{lem:ls}, we have
    \[
        \limsup_{j \to \infty} F_{\varepsilon_j}(u_j, \Omega)
            \leq \limsup_{j \to \infty} F_{\varepsilon_j}(v_j, \Sigma_\rho)
            \leq \int_\Sigma \psi(\nu_u) d\Hausdorff^{N-1} + \sigma.
    \]
    Since we can construct such a sequence for each $\sigma > 0$, we can diagonalize this family to get a sequence $\set{u_j}$ such that
    \[
        \limsup_{j \to \infty} F_{\varepsilon_j}(u_j, \Omega)
            \leq \int_\Sigma \psi(\nu_u) d\Hausdorff^{N-1},
    \]
    concluding the proof.
\end{proof}

\section{Acknowledgments}

The author would like to thank Irene Fonseca and Giovanni Leoni for insightful and indispensable conversations in the preparation of this paper.
This work was completed in partial fulfillment of the requirements for the degree of Doctor of Philosophy in Mathematics under the direction of I. Fonseca and G. Leoni at Carnegie Mellon University.
This work was partially funded by the National Science Foundation under Grant No. DMS-2108784 ``Mathematics of Microstructure in Origami, Robotics, and Electrochemistry.''

\printbibliography

\end{document}